\documentclass[10pt,journal]{IEEEtran}   	% use "amsart" instead of "article" for AMSLaTeX format
\usepackage{graphicx}				% Use pdf, png, jpg, or epsÃÂ§ with pdflatex; use eps in DVI mode
\usepackage[cmex10]{amsmath}
\usepackage{amssymb}
\usepackage{amsthm}
\usepackage{bm}
\usepackage{color}

\theoremstyle{plain}
\newtheorem{thm}{Theorem}
\newtheorem{lem}{Lemma}
\newtheorem{prop}{Proposition}
\newtheorem{cor}{Corollary}

\theoremstyle{definition}
\newtheorem{defn}{Definition}

\theoremstyle{remark}
\newtheorem{rem}{Remark}

\newcommand{\sref}[1]{Section~\ref{#1}}
\newcommand{\abs}[1]{\left|#1\right|}

\newcommand{\diag}{\mathop{\mathrm{diag}}}

\usepackage{amsfonts}

\title{Explicit Conditions on Existence and Uniqueness of Load-Flow Solutions in Distribution Networks}

\author{Cong~Wang,~\IEEEmembership{Student Member,~IEEE}, Andrey~Bernstein,~\IEEEmembership{Member,~IEEE}, Jean-Yves~Le~Boudec,~\IEEEmembership{Fellow,~IEEE}, and Mario~Paolone,~\IEEEmembership{Senior Member,~IEEE}

\thanks{The authors are with EPFL, Lausanne, Switzerland.} \vspace{-0.8cm}}

\begin{document}
\maketitle

\begin{abstract}
%In this paper, we develop a novel iterative load flow method that solves non-linear power flow equations for balanced distribution networks with shunt elements. Besides, explicit sufficient conditions that guarantee the existence and uniqueness of the feasible load flow solution are provided.
%The method has low computational complexity and the conditions can be efficiently verified in a real system. As a special case, for \emph{radial} networks, the complexity is \emph{linear} in the number of nodes. Thus the proposed approach is of particular interest in the modern microgrid setup in the context of real-time control. The theory has been verified numerically in IEEE models.
We present explicit sufficient conditions that guarantee the existence and uniqueness of the feasible load-flow solution for distribution networks with a generic topology (radial or meshed) modeled with positive sequence equivalents. In the problem, we also account for the presence of shunt elements. The conditions have low computational complexity and thus can be efficiently verified in a real system. Once the conditions are satisfied, the unique load-flow solution can be reached by a given fixed point iteration method of approximately linear complexity.
%As a special case, for \emph{radial} networks, the complexity is \emph{linear} in the number of nodes.
Therefore, the proposed approach is of particular interest for modern active distribution network (ADN) setup in the context of real-time control. The theory has been confirmed through numerical experiments.
\end{abstract}
\vspace*{-0.2cm}
\begin{IEEEkeywords}
load flow solution, fixed point method, existence and uniqueness, distribution networks.%\vspace{-0.4cm}
\end{IEEEkeywords}
%\vspace*{-0.3cm}
\section*{Nomenclature}
\begin{tabular}{rl}
$v = (v_1, v_2, ..., v_N)^T$ & $v_k$ is the positive-sequence \\
& complex voltage at bus $k$. \\
$i = (i_1,i_2,...,i_N)^T$ & $i_k$ is the positive-sequence  \\
& complex nodal current of bus $k$. \\
$s = (s_1, s_2, ..., s_N)^T$ & $s_k$ is the complex nodal\\
& power injected into bus $k$. \\
bus $0$ & Slack bus, with $v_0$=1 p.u. \\
%& with complex voltage $=1$. \\
$i_0,s_0$ & Slack bus complex nodal \\
& current and power. \\
%$s_0$ & Slack bus complex nodal \\
%& power. \\
$Y$ & Positive-sequence nodal \\
& admittance matrix. \\
$Y_{LL}$ & Square submatrix of $Y$, \\
& omitting the slack bus. \\
$w_i$, $i=1,...,N$ & Positive-sequence complex voltage  \\
& at node $i$ when $s$ is a zero vector. \\
$W = \diag(w_i)$\\
$u = W^{-1}v$ & Normalized node voltages. \\
(For any $z$ in $\mathbb{C}$) $\overline{z}$ & The complex conjugate of $z$.
\end{tabular}

\section{Introduction}
%Motivate the problem by traditional stuff + commelec
\IEEEPARstart{T}{he} load-flow problem, which expresses the link between complex node voltages and complex nodal power injections, is one of the main tasks in power system theory and applications. In the context of distribution networks, it is especially interesting to consider the case where non-slack buses are $PQ$ buses. In this paper we consider a network with one single slack bus, at which the complex voltage is assumed fixed and known, while the rest are $PQ$ buses. Given a vector of nodal power injections into $PQ$ buses, the problem is then to compute the vector of complex node voltages in the network that is \emph{feasible} (i.e., close to $1$ p.u. in magnitude). In the rest of the manuscript, we make reference to the load-flow problem formulated for the positive sequence.
%
%
%t .
%The problem is formulated using non-linear power flow equations that describe the connection between the various power system quantities, such as voltage phasors, nodal current phasors, and power injections. In distribution networks, an instance of this problem, namely the \emph{load-flow problem}, is of particular interest. In this problem, the network is assumed to have a single slack bus, while the rest are $PQ$ buses. Given a vector of nodal power injections of the $PQ$ buses, the problem is to compute the vector of complex node voltages in the network that is \emph{feasible} (namely, close to $1$ p.u. in magnitude).
%
%Power flow analysis is one of the main tasks in power system theory  and applications.
%The problem is formulated using non-linear power flow equations that describe the connection between the various power system quantities, such as voltage phasors, nodal current phasors, and power injections. In distribution networks, an instance of this problem, namely the \emph{load-flow problem}, is of particular interest. In this problem, the network is assumed to have a single slack bus, while the rest are $PQ$ buses. Given a vector of nodal power injections of the $PQ$ buses, the problem is to compute the vector of complex node voltages in the network that is \emph{feasible} (namely, close to $1$ p.u. in magnitude).

Due to the non-linearity of the equations, the existence and uniqueness of the solution to the load-flow problem is not guaranteed in general \cite{MultiSol1}, \cite{MultiSol2}, \cite{MultiSoluAlbert}. There is extensive literature on the subject as detailed in Section \ref{sec:soa}. But for grid control, in order to maintain the system in feasible electrical states, it is essential to provide conditions guaranteeing that the implemented power setpoint leads to the unique feasible solution of the load-flow problem. Specifically, in active distribution networks (and particularly, microgrids), these conditions are further expected to be both explicitly formulated and verifiable in real-time.

%Developing \emph{explicit} conditions that can be evaluated in \emph{in real-time} is especially important in the area of active distribution networks (and in particular, microgrids).
%Recently, in the context of active distribution networks (and in particular, microgrids), there is a need for developing explicit conditions for existence and uniqueness of the solution to the load-flow problem that can be evaluated in \emph{in real-time}.
%The problem of evaluating the existence and uniqueness of the solution to the load-flow problem \emph{in real-time} has become even more important in the context of active distribution networks, and in particular, microgrids.
There are multiple scenarios that have such expectations. One typical case is related to the \emph{islanding maneuver}, namely the disconnection from the main grid due to an intentional or non-intentional decision (e.g., \cite{island-Concep}). In particular, with respect to the non-intentional islanding, there is a need to evaluate in real-time whether a given resource can serve as a slack for the islanded microgrid \cite{IslandingCommelec}. This evaluation is based on verifying whether the currently implemented setpoint leads to the unique feasible solution of the corresponding load-flow problem.
%The availability of an efficient procedure for testing whether the current power setpoint leads to the unique solution of the corresponding load-flow problem is thus required.
%The islanding maneuver is a particular application in a more general setting of real-time control of active distribution networks.
Another practical example is related to the recently introduced framework for performing real-time control of active distribution networks using explicit power setpoints \cite{commelec1}.
%Recently, a new framework for performing this task by using explicit power setpoints, called Commelec, was introduced \cite{commelec1}.
%Recently, in the context of active distribution networks, this problem has become even more important due to the introduction of a new framework for the real-time control of grids using explicit power setpoints called Commelec \cite{commelec1}.
In this framework, the knowledge of the current system state (obtained  via a corresponding state estimation procedure) is assumed. A typical task in this framework is then to decide whether a given collection of power setpoints is \emph{admissible} in the sense that the application of these setpoints will result in a feasible voltage profile of the grid. Hence as we can see from these situations, the research in this paper is of practical significance.

In the paper, we give explicit conditions that guarantee the existence and uniqueness of the load-flow solution for (possibly meshed) distribution networks with shunt elements. The unique solution can be reached by an iterative load-flow method given in this paper. Our conditions depend on the current state of the grid as well as on the requested power setpoints. The proposed approach is computationally efficient, with approximately linear complexity.  Hence it can be applied in a real-time control framework. We also provide conditions in the ``classical'' setup, where the knowledge of the current grid state is absent. In this case, we show that our results are stronger than those introduced so far in the literature. Note that it is possible to extend our results to more general three-phase distribution networks, but this is the subject of ongoing work.

The paper is structured as follows. In Section \ref{sec:soa}, we review the related work. In Section \ref{sec:pb}, we present the load-flow problem and its useful equivalent formulation as a fixed point problem. In Section \ref{sec:res}, we give our main result, and prove it in Section \ref{sec:proofs}. In Section \ref{sec:num}, we provide numerical evaluation of our method. Finally, we conclude in Section \ref{sec:conc}.

\section{Related Work} \label{sec:soa}
In the last few decades, the existence and uniqueness of the solution to the load-flow problem have been studied from various perspectives.

In \cite{MIlic1}, conditions for the existence and uniqueness of the solution to reactive power-voltage magnitude problem are given and analyzed. Based on \cite{MIlic1}, \cite{MIlic2} extends the result to active power-voltage angle problem. Under certain assumptions, by decoupling the active and reactive power (i.e., considering a sub-problem of active power with voltage angle, and a sub-problem of reactive power with voltage magnitude), sufficient conditions for load-flow solvability are explored. For balanced radial distribution networks, the uniqueness of a feasible load-flow solution is proved by exploiting the radial structure in \cite{Chiang1}. In \cite{Chiang2}, the result is extended to the unbalanced radial three-phase distribution networks. However, all these results are based on certain assumptions and cannot be generically applied.

Recently, the focus has been moved to fixed point load-flow analysis since the fixed point theorem can guarantee the uniqueness of the load-flow solution. In fact, the first attempt of applying fixed point theorem to power systems dates back to \cite{Conv1}, which focused on the study of convergence property of the Newton method. For the latest research, in \cite{Lisboa}, an efficient fixed point load-flow method is presented for radial distribution networks, but there is no further discussion about the convergence and solvability. Later, in \cite{Bolo}, another form of fixed point load-flow method is proposed for distribution network with single slack bus. In the same paper, sufficient conditions are given to guarantee the existence and uniqueness of solution. These sufficient conditions are improved in \cite{ImproveBolo}.

In this paper, we use a fixed point formulation of the load-flow problem; then we specify a domain around a feasible point and provide sufficient conditions that guarantee the existence and uniqueness of load-flow solution in this domain. Under the proposed conditions, the unique solution can be reached using the fixed point iteration. It should be noticed that, by this approach, the feasibility of load-flow solution is usually preserved. %Besides, by iterative procedure, the unique load flow solution can be reached efficiently.

The theory proposed here shares some similarities with the fixed point load-flow methods established in \cite{Lisboa}, \cite{Bolo} and \cite{ImproveBolo}. But, the method in \cite{Lisboa} is a special case of this paper. Furthermore, the sufficient conditions in this paper are more general than the conditions in \cite{Bolo} and \cite{ImproveBolo}, and thus improve these results.

%@Article{Conv2,
%		AUTHOR = "F. F. Wu",
%		TITLE = "Theoretical study of the convergence of the fast decoupled load flow",
%		JOURNAL = "IEEE Trans. on Power App. and Syst.",
%		VOLUME = "PAS-96",
%		NUMBER = "1",
%		PAGES = "268-275",
%		MONTH = jan,
%		YEAR = "1977"}

\section{The Load Flow Problem} \label{sec:pb}
%As mentioned, this paper is focused on single-phase distribution networks.
%However, as we show later \textbf{where?}, the results can be smoothly extended to the three-phase case with proper device modeling. This extension is a subject of an ongoing work.

%But as can be seen later, the result can be smoothly extended to three-phase distribution networks with proper device modeling. The extension will be discussed in details in our next-stage work, which is expected to come up soon.

We consider a distribution network modeled by its positive sequence equivalents with $N$ $PQ$ buses and one slack bus (in essence, a $V \theta$ bus). Without loss of generality, we assume that the complex voltage of the slack bus is $1$ p.u. Let $v = (v_1, v_2, ..., v_N)^T$ denote the vector of complex node voltages of the $PQ$ buses, $i = (i_1,i_2,...,i_N)^T$ denote the vector of complex nodal currents into the $PQ$ buses, $i_0$ denote the complex nodal current into the slack bus, $s = (s_1, s_2, ..., s_N)^T$ denote the vector of complex nodal powers injected into the $PQ$ buses (negative value in real or imaginary part means consumed), and $s_0$ denote the complex nodal power injected into the slack bus. Also, for any complex number $z$, we denote its complex conjugate by $\overline{z}$. A similar notation holds for vectors and matrices.

As known, the nodal powers and nodal currents can be expressed in matrix form as
\begin{equation} \label{eqn:pfe1}
\left[
\begin{array}{c}
\overline{s}_{0}\\
\overline{s}
\end{array}
\right]
=
\left[
\begin{array}{cc}
1 &\\
&\diag(\overline{v})
\end{array}
\right]
\left[
\begin{array}{c}
i_{0}\\
i
\end{array}
\right],
\end{equation}
%and
\begin{equation} \label{eqn:pfe2}
\left[
\begin{array}{c}
i_{0}\\
i
\end{array}
\right]
=Y
\left[
\begin{array}{c}
1\\
v
\end{array}
\right].
\end{equation}
Here, $Y$ is the $(N+1)\times(N+1)$ nodal admittance matrix of the system.

The classical \emph{load-flow problem} in this setup is defined as follows: Given the nodal powers $s$, solve the set of equations \eqref{eqn:pfe1} and \eqref{eqn:pfe2} to obtain the nodal voltages $v$ and the power at the slack bus $s_0$.
The nodal voltages are generally required to be \emph{feasible} in the sense that all the node voltages have magnitude close to 1 p.u.

In this paper, we rely on an equivalent formulation of this problem that is known as \emph{implicit $Z_{bus}$ formulation}, see e.g., \cite{Zbus}.
%We next show that this problem can be formulated in a useful equivalent way in order to develop a solution method with theoretical guarantees.
First, partition the admittance matrix $Y$ as
\begin{equation} \label{eqn:Yll}
Y=
\left[
\begin{array}{cc}
Y_{00} & Y_{0L}\\
Y_{L0} & Y_{LL}
\end{array}
\right],
\end{equation}
where $Y_{00}$ is a number, $Y_{0L}$ is a $1\times N$ row vector, $Y_{L0}$ is an $N\times1$ column vector, $Y_{LL}$ is an $N\times N$ matrix. Now, we claim that $Y_{LL}$ is an invertible matrix. This fact was mentioned, e.g., in \cite{Bolo}, without a proof; in Appendix \ref{sec:app} we give a proof that covers a broad range of distribution networks.
The implicit $Z_{bus}$ formulation is then given by the following proposition; for completeness, we also provide a short proof.
%Our method is based on the following proposition, which is called implicit $Z_{bus}$ formulation in \cite{Zbus} without proof.
\begin{prop} \label{prop:equiv_LF}
The solution $v$ to the original load-flow problem can be found by solving the following fixed point equation
\begin{equation} \label{eqn:equiv_PFE}
v = w + Y_{LL}^{-1} \diag(\overline{v})^{-1}\overline{s}\triangleq G(v),
\end{equation}
\begin{equation} \label{eqn:zero_load}
\text{where} \qquad\qquad w \triangleq -Y_{LL}^{-1}Y_{L0}\hspace{3cm}
\end{equation}
is given and is equal to the vector of complex voltages when power injections are zero (\emph{zero-load voltage} of the grid).
\end{prop}

\begin{proof}
By \eqref{eqn:pfe2} and \eqref{eqn:Yll}, we have that
\[
i = Y_{L0} + Y_{LL}v.
\]
Thus, clearly, $w \triangleq -Y_{LL}^{-1}Y_{L0}$ is the zero-load voltage of the grid. From \eqref{eqn:pfe1},
\[
i = \diag(\overline{v})^{-1}\overline{s}
\]
\[
\text{and hence} \qquad\quad Y_{LL}^{-1}\diag(\overline{v})^{-1}\overline{s}= -w + v, \hspace{6cm}
\]
which completes the proof.
\end{proof}
\begin{rem}
 This formulation can be viewed as a direct result of the superposition theorem: $v$ is the superposition of the voltages $w$, resulting from current injections by the slack bus when all other injections are absent ($s_i=0$, $i=1...N$) plus the voltages resulting from current injections due to $s$ when the slack bus injection is absent.
\end{rem}

In the subsequent sections, we propose and prove sufficient conditions under which there exists a unique feasible solution to \eqref{eqn:equiv_PFE}, which can be found by the iteration
\begin{equation} \label{eqn:iter_v}
v^{(k+1)} = w + Y_{LL}^{-1} \diag(\overline{v}^{(k)})^{-1}\overline{s}.
\end{equation}
%Mathematically, it is useful to parametrize \eqref{eqn:equiv_PFE} in a different way. Let $W \triangleq \diag (w)$, and $u \triangleq W^{-1} v$ denote the \emph{normalized} voltage. Then, it is easy to see that \eqref{eqn:equiv_PFE} is equivalent to
%\begin{equation} \label{eqn:equiv_PFE_u}
%u = \mathbbm{1} + W^{-1}Y_{LL}^{-1}\overline{W}^{-1}\diag(\overline{u})^{-1}\overline{s},
%\end{equation}
%where $\mathbbm{1}=(1,1,...,1)^T$ is the unity vector.
%Clearly, uniqueness of the solution $u$ for \eqref{eqn:equiv_PFE_u} guarantees uniqueness of $v = W u$ for \eqref{eqn:equiv_PFE}.

\section{Main Result} \label{sec:res}
In this section, we give conditions on the complex power injections $s$ %and on the initial $v^{(0)}$
which guarantee that iteration \eqref{eqn:iter_v} converges to the unique feasible solution $v$ of the load-flow problem.
We also provide computational complexity of the method.

Before presenting our method formally, we give a high-level outline. First, we assume the knowledge of a pair $(\hat{v},\hat{s})$ that satisfies the load-flow equations \eqref{eqn:equiv_PFE}. This pair can be interpreted as the current (actual) state of the grid obtained via a measurement and state estimation process. In addition, we are given a desired ``next'' power setpoint $s$. Our conditions are thus formulated in terms of $(\hat{v},\hat{s})$ and $s$, and guarantee the unique feasible solution $v$ to \eqref{eqn:equiv_PFE} which is ``close'' to $\hat{v}$. Finally, we provide conditions on the starting point $v^{(0)}$ from which this solution can be computed using iteration \eqref{eqn:iter_v}.

As mentioned in the introduction, such a procedure is especially useful in the modern ADN setup, where the electrical state is continuously estimated and is varying slowly from its current value. In case there is no knowledge of the current state, a trivial choice for $(\hat{v},\hat{s})$ is $(w, \textbf{0})$, where $w$ is the zero-load voltage profile \eqref{eqn:zero_load}. For details, see Corollary \ref{cor:zero_load} below.

\subsection{Main Theorem}
%\lc{explain main theorem in the $V$ space, not $u$}
%
%\lc{e.g. $\mathcal{D}=\{v: \abs{v_i-v_i^0}\leq \rho\abs{w_i}\}$}

We introduce some further notation. Let $W \triangleq \diag (w)$ and %let $u \triangleq W^{-1} v$ denote the normalized node voltage vector. Also,
set
\begin{equation}
\xi(s) \triangleq \| W^{-1}Y_{LL}^{-1}\overline{W}^{-1}\diag(\overline{s}) \|_{\infty},
\end{equation}
where, for any complex matrix $A$,
$$\|A \|_{\infty} \triangleq \max_i \sum_j |A_{ij}|$$
denotes the matrix norm induced by the $\ell_{\infty}$ norm. Let
\begin{equation} \label{eqn:umin}
u_{min} \triangleq \min_{j} \left | \hat{v}_j / w_j \right |.
\end{equation}
Below is our main result. Its proof is in Section \ref{sec:proofs}.
\begin{thm} \label{thm:main}
Let the pair $(\hat{v},\hat{s})$ be a known solution to the load-flow problem \eqref{eqn:equiv_PFE}. Consider some other candidate complex power injection $s$. Assume that
\begin{equation} \label{eqn:cond1}
\xi(\hat{s}) < u_{min}^2
\end{equation}
\begin{equation} \label{eqn:cond2}
\text{and}\qquad\Delta \triangleq \left(u_{min}-\frac{\xi(\hat{s})}{u_{min}}\right)^2-4\xi(s-\hat{s}) > 0.
\end{equation}
Then there exists a unique solution $v$ to the load-flow problem, such that the pair $(v, s)$ satisfies \eqref{eqn:equiv_PFE} and $v$ belongs to
\[
\mathcal{D} \triangleq \{v: \abs{v_i-\hat{v}_i}\leq \rho\abs{w_i}\}
\]
\[
\text{with}\qquad\qquad\rho \triangleq  \frac{\left(u_{min}-\frac{\xi(\hat{s})}{u_{min}}\right)-\sqrt{\Delta}}2.\hspace{6cm}
\]
Moreover, this solution can be reached using the iterative procedure \eqref{eqn:iter_v} by starting with any $v^{(0)} \in \mathcal{D}$.
\end{thm}

\iffalse
  Observe that Theorem \ref{thm:main} can be used in the following way. Given the estimated current state $(\hat{v},\hat{s})$ and the desired ``next'' injected power $s$, verify that conditions \eqref{eqn:cond1} and \eqref{eqn:cond2} hold. If so, use iteration \eqref{eqn:iter_v} to obtain the corresponding solution $v$. Such a procedure is especially useful in the modern microgrid setup, where the electrical state is continuously estimated and is varying slowly from its current value. In case there is no knowledge of the current state, the following corollary can be used.
\fi
In case there is no knowledge of the current state $(\hat{v},\hat{s})$, the following corollary can be used.
\begin{cor} \label{cor:zero_load}
Suppose that the complex power $s$ satisfies $\xi(s) < 0.25$. Then, there exists a unique solution $v$
to the load-flow problem, such that the pair $(v, s)$ satisfies \eqref{eqn:equiv_PFE} and $v$ belongs to
\[
\mathcal{D}' \triangleq \left \{v: \abs{v_i-w_i}\leq \frac{(1 - \sqrt{1 - 4\xi(s)})\abs{w_i}}2 \right \}.
\]
This solution can be reached using the iterative procedure in \eqref{eqn:iter_v} by starting with any $v^{(0)} \in \mathcal{D}'$.
\end{cor}
 \begin{proof}
 We use Theorem \ref{thm:main} with the choice $\hat{v} = w$ and $\hat{s} = \textbf{0}$. In this case, as $\xi(\textbf{0}) = 0$, condition \eqref{eqn:cond1} is always satisfied. Also, as $u_{min} = 1$, condition \eqref{eqn:cond2} becomes $\xi(s) < 0.25$ and $\rho$ is given by $(1 - \sqrt{1 - 4\xi(s)})/2$.
 \end{proof}
We demonstrate the numerical utility of choosing either the method of Theorem \ref{thm:main} or that of Corollary \ref{cor:zero_load} in Section \ref{sec:num}.

\subsection{Comparison with Existing Results}
In \cite{Bolo}, the following sufficient condition for the unique solution of the load-flow problem was given: $\exists p\in [1,\infty]$ and $q = p/(p - 1)$ such that
\begin{equation} \label{eqn:bol}
\| W^{-1}Y_{LL}^{-1}\overline{W}^{-1} \|_{p}^{*}\| s \|_{q}<0.25,
\end{equation}
where, for any matrix $A$, $\|A \|^*_p \triangleq  \max_{h} \|A_{h \bullet }\|_p$, and the notation $A_{h\bullet}$ stands for the $h$-th row of $A$. This work has been improved in \cite{ImproveBolo} as follows: $\exists p\in [1,\infty]$, $q = p/(p - 1)$,  and a real-valued diagonal matrix $\Lambda$ such that
\begin{equation} \label{eqn:yu}
\| W^{-1}Y_{LL}^{-1}\overline{W}^{-1}\Lambda \|_{p}^{*}\| \Lambda^{-1}s \|_{q}<0.25.
\end{equation}
We next show that our condition is weaker (thus the result is stronger). Since no knowledge of the current electrical state is assumed in both \cite{Bolo} and \cite{ImproveBolo}, we compare it with the condition of Corollary \ref{cor:zero_load}. By Holder's inequality with $\frac{1}{p}+\frac{1}{q}=1$,
\begin{small}
\begin{align}
  \begin{split}
    &\xi(s)=\| W^{-1}Y_{LL}^{-1}\overline{W}^{-1}\diag(\overline{s}) \|_\infty\\
    = & \max_{i} \sum_{j} \abs{{\left( W^{-1}Y_{LL}^{-1}\overline{W}^{-1}\Lambda \right)}_{ij}}\abs{{\left( \Lambda^{-1}s \right)}_j}\\
    = & \sum_{j} \abs{{\left( W^{-1}Y_{LL}^{-1}\overline{W}^{-1}\Lambda \right)}_{i_{max}j}}\abs{{\left( \Lambda^{-1}s \right)}_{j}}\\
    \leq & \| {\left( W^{-1}Y_{LL}^{-1}\overline{W}^{-1}\Lambda \right)}_{i_{max}\bullet} \|_{p}\| \Lambda^{-1}s \|_{q}\\
     \leq & \| W^{-1}Y_{LL}^{-1}\overline{W}^{-1}\Lambda \|_{p}^{*}\| \Lambda^{-1}s \|_{q}.
  \end{split}
\end{align}
\end{small}
Thus, whenever \eqref{eqn:bol} or \eqref{eqn:yu} is satisfied, we have that $\xi(s) < 0.25$, hence the hypothesis of Corollary \ref{cor:zero_load} is satisfied. We complement this result in  \sref{sec-illu} by showing that the converse is not true.
\vspace*{-0.25cm}
\subsection{Computational Complexity}
%\subsubsection{The complexity of one iteration}
\textit{1) The complexity of one iteration}: In general, each iteration of \eqref{eqn:iter_v} can be computed either directly or through solving linear equations. Such procedures usually require $O(N^2)$ computational complexity for a general linear system. But our experience shows the computational complexity can approximately be $O(N)$ if using LU decomposition with complete Markowitz pivoting \cite{Markowitz}. This is because the nodal admittance matrices are structurally sparse and symmetric in general, for which the pivoting reduces the number of fill-ins and preserve the sparsity in LU decomposition \cite{LU1}, \cite{LU2}.

For radial distribution networks, a similar decomposition is given in \cite{Lisboa} by exploiting the grid structure from a graph-theoretic perspective. Such decomposition guarantees $O(N)$ computational complexity for these cases under proper hypothesis.

%\subsubsection{The complexity of checking conditions}
\textit{2) The complexity of checking conditions}: Generally, complexity of checking conditions is mainly the complexity of computing $\xi(\hat{s})$ and $\xi(s-\hat{s})$, which is $O(N^2)$. But for networks where the decomposition in \cite{Lisboa} applies, this complexity can be reduced to $O(N)$ by only computing and comparing the rows that correspond to \emph{leaf} nodes.

\section{Proof of Theorem \ref{thm:main}} \label{sec:proofs}
For the purpose of the proof, we find it useful to parametrize \eqref{eqn:equiv_PFE} in a different way. Let $u \triangleq W^{-1} v$ denote the \emph{normalized} voltage with respect to an unloaded grid. Then, it is easy to see that \eqref{eqn:equiv_PFE} is equivalent to
\begin{equation} \label{eqn:equiv_PFE_u}
u = \bm{1} + W^{-1}Y_{LL}^{-1}\overline{W}^{-1}\diag(\overline{u})^{-1}\overline{s}\triangleq \tilde{G}(u),
\end{equation}
where $\bm{1}=(1,1,...,1)^T$ is the unity vector.
Clearly, any conditions on $u$ provide corresponding conditions on $v$ using the invertible mapping $v = W u$. We thus perform the analysis of \eqref{eqn:equiv_PFE_u} and the corresponding iteration
\begin{equation} \label{eqn:iter_u}
u^{(k+1)} = \bm{1} + W^{-1}Y_{LL}^{-1}\overline{W}^{-1}\diag(\overline{u}^{(k)})^{-1}\overline{s}.
\end{equation}

From the Banach fixed point theorem \cite{Banach}, if the operator $\tilde{G}$ is a contraction mapping on a metric space $(\tilde{\mathcal{D}},\tilde{d})$,
%with $\tilde{\mathcal{D}} \subset \mathbb{C}^{N}$
then there is a unique fixed point $u^{*}$ in $\tilde{\mathcal{D}}$. Moreover, $u^{*}$ can be reached by iterative update of $u^{(k+1)}=\tilde{G}(u^{(k)})$ from an arbitrary $u^{(0)}$ in $\tilde{\mathcal{D}}$. In the rest of this section, we show that under the conditions of Theorem \ref{thm:main}, operator $\tilde{G}$ is a contraction mapping in the sense that (i) $\tilde{G}$ is a self-mapping of $u$ on a closed set $\tilde{\mathcal{D}}$
% with $\ell_\infty$ norm
, and (ii) $\tilde{G}$ has the contraction property: $\|\tilde{G}(u^2)-\tilde{G}(u^1)\|_\infty<\|u^2-u^1\|_\infty$ for any $u^1,u^2\in\tilde{\mathcal{D}}$.
\iffalse
Mathematically, by plugging in $W = \diag (w)$ and $u = W^{-1} v$, \eqref{eqn:equiv_PFE} and \eqref{eqn:iter_v} can be written in a different way as following in \eqref{eqn:equiv_PFE_u} and \eqref{eqn:iter_u}.
\begin{equation} \label{eqn:equiv_PFE_u}
u = \bm{1} + W^{-1}Y_{LL}^{-1}\overline{W}^{-1}\diag(\overline{u})^{-1}\overline{s}\triangleq G_u(u),
\end{equation}
\begin{equation} \label{eqn:iter_u}
u^{(k+1)} = \bm{1} + W^{-1}Y_{LL}^{-1}\overline{W}^{-1}\diag(\overline{u}^{(k)})^{-1}\overline{s},
\end{equation}
Here, $\bm{1}=(1,1,...,1)^T$ is the unity vector.
Clearly, this new formulation is equivalent to the original one because $W$ is an invertible diagonal matrix. We then use it to help complete the proof.

From Banach fixed point theorem, if operator $G$ is a contraction mapping from a closed set $\mathcal{D} \subset \mathbb{C}^{N}$ to itself, then there is a unique fixed point $v^{*}$ in $\mathcal{D}$. Moreover, $v^{*}$ can be reached by iterative update of $v^{(k+1)}=G(v^{(k)})$ from an arbitrary $v^{(0)}$ in $\mathcal{D}$. Therefore, the key step to take is obtain conditions under which $G$ is a self-mapping and contraction.
\fi
\subsection{Proof of self-mapping}
\begin{lem}
Suppose that the pair $(\hat{v},\hat{s})$ and the complex power $s$ satisfy \eqref{eqn:cond1}
%$$\xi(\hat{s}) < u_{min}^2$$
and \eqref{eqn:cond2}.
%$$\Delta = \left(u_{min}-\frac{\xi(\hat{s})}{u_{min}}\right)^2-4\xi(s-\hat{s}) > 0.$$
Then $\tilde{G}$ is a self-mapping of $u$ on
\begin{equation} \label{eqn:D_tilde}
\tilde{\mathcal{D}} \triangleq \{u: \abs{u_i-\hat{u}_i}\leq \rho\}
\end{equation}
\[
\text{with\qquad\qquad}\rho = \frac{\left(u_{min}-\frac{\xi(\hat{s})}{u_{min}}\right)-\sqrt{\Delta}}2\hspace{3cm}
\]
and $\hat{u}_i=\hat{v}_i / w_i$.
\end{lem}
\begin{proof}
Since $(\hat{v},\hat{s})$ satisfies the power flow equation \eqref{eqn:equiv_PFE}, we have that $\hat{u}=\bm{1}+W^{-1}Y_{LL}^{-1}\overline{W}^{-1}\diag(\overline{\hat{u}})^{-1}\overline{\hat{s}}$ in addition to \eqref{eqn:equiv_PFE_u}.
Thus,
\begin{small}
\begin{equation}
\tilde{G}(u)-\hat{u}=W^{-1}Y_{LL}^{-1}\overline{W}^{-1}\left(\diag(\overline{u})^{-1}\overline{s}-\diag(\overline{\hat{u}})^{-1}\overline{\hat{s}}\right). \nonumber
\end{equation}
\end{small}
%\begin{align}
%  \begin{split}
%    & \tilde{G}(u)-\hat{u}\\
%    = & W^{-1}Y_{LL}^{-1}\overline{W}^{-1}\left(\diag(\overline{u})^{-1}\overline{s}-\diag(\overline{\hat{u}})^{-1}\overline{\hat{s}}\right).
%  \end{split}
%\end{align}
%If $\abs{u_{i}-\hat{u}_{i}} \leq r < u_{min}$, then for all $i$, we have
%
%\begin{align}
%=======
Our goal is to show that there exists a radius $r$ such that if $\abs{u_{i}-\hat{u}_{i}} \leq r$ then $\abs{\tilde{G}(u)_{i}-\hat{u}_{i}}\leq r$ for all $i$.
We have

\begin{small}
\begin{align*}
  \begin{split}
    &\abs{{\tilde{G}(u)}_{i}-\hat{u}_{i}}
    = \abs{\sum_{j} {\left(W^{-1}Y_{LL}^{-1}\overline{W}^{-1}\right)}_{ij} \left(\frac{\overline{s}_j}{\overline{u}_j}-\frac{\overline{\hat{s}}_j}{\overline{\hat{u}}_{j}}\right)}\\
    \leq & \sum_{j} \abs{{\left(W^{-1}Y_{LL}^{-1}\overline{W}^{-1}\right)}_{ij}\frac{\overline{\hat{s}}_j(\overline{\hat{u}}_j-\overline{u}_j)+\overline{\hat{u}}_j(\overline{s}_j-\overline{\hat{s}}_j)}{\overline{u}_j\overline{\hat{u}}_j}}\\
    \leq & \sum_{j} \abs{{\left(W^{-1}Y_{LL}^{-1}\overline{W}^{-1}\right)}_{ij}\overline{\hat{s}}_j\frac{(\overline{\hat{u}}_j-\overline{u}_j)}{\overline{u}_j\overline{\hat{u}}_j}}\\
    & +\sum_{j} \abs{{\left(W^{-1}Y_{LL}^{-1}\overline{W}^{-1}\right)}_{ij}\frac{(\overline{s}_j-\overline{\hat{s}}_j)}{\overline{u}_j}}. \nonumber
  \end{split}
\end{align*}
\end{small}

%Consider the fact that $\abs{\hat{u}_j}\geq u_{min}$ and $\abs{\overline{\hat{u}}_j-\overline{u}_j}\leq r$, we have $\abs{\overline{u}_j}\geq u_{min}-r$. So that
Now, assume that $\abs{u_{i}-\hat{u}_{i}} \leq r < u_{min}$, where $u_{min}$ is given in \eqref{eqn:umin}. Also, by the definition of $u_{min}$, we have that
$\abs{\hat{u}_j}\geq u_{min}$. Therefore, $\abs{\overline{u}_j}\geq u_{min}-r$, and
\begin{scriptsize}
\begin{align}
  \begin{split}
    & \sum_{j} \abs{{\left(W^{-1}Y_{LL}^{-1}\overline{W}^{-1}\right)}_{ij}\overline{\hat{s}}_j\frac{(\overline{\hat{u}}_j-\overline{u}_j)}{\overline{u}_j\overline{\hat{u}}_j}}\\
    \leq & \sum_{j} \abs{{\left(W^{-1}Y_{LL}^{-1}\overline{W}^{-1}\right)}_{ij}\overline{\hat{s}}_j}\frac{r}{(u_{min}-r)u_{min}}
    \leq \frac{\xi(\hat{s})r}{(u_{min}-r)u_{min}}. \nonumber
  \end{split}
\end{align}
\end{scriptsize}
Similarly,
\begin{scriptsize}
\begin{align}
  \begin{split}
    & \sum_{j} \abs{{\left(W^{-1}Y_{LL}^{-1}\overline{W}^{-1}\right)}_{ij}\frac{(\overline{s}_j-\overline{\hat{s}}_j)}{\overline{u}_j}}\\
    \leq & \sum_{j} \abs{{\left(W^{-1}Y_{LL}^{-1}\overline{W}^{-1}\right)}_{ij}(\overline{s}_j-\overline{\hat{s}}_j)}\frac{1}{(u_{min}-r)}
    \leq \frac{\xi(s-\hat{s})}{(u_{min}-r)}. \nonumber
  \end{split}
\end{align}
\end{scriptsize}
Combine them and obtain
\begin{equation}
\abs{\tilde{G}(u)_{i}-\hat{u}_{i}}\leq \frac{\xi(\hat{s})r}{(u_{min}-r)u_{min}}+ \frac{\xi(s-\hat{s})}{(u_{min}-r)}.
\end{equation}
Therefore, we have a self-mapping if
\begin{equation} \label{eqn:selfmap1}
\frac{\xi(\hat{s})r}{(u_{min}-r)u_{min}}+ \frac{\xi(s-\hat{s})}{(u_{min}-r)} \leq r.
\end{equation}
It can be re-organized as
\begin{equation} \label{eqn:selfmap2}
r^2-\left(u_{min}-\frac{\xi(\hat{s})}{u_{min}}\right)r+\xi(s-\hat{s})\triangleq f(r) \leq 0.
\end{equation}
We thus have shown that %$\tilde{G}(\tilde{D}(r))\subseteq \tilde{D}(r)$
$\tilde{G}$ is a self-mapping if there exists an $r\in(0,u_{min})$ such that $f(r)\leq 0$. Since $f(r)$ is a convex polynomial of degree two and $f(0)=\xi(s-\hat{s})>0$, we know there is an interval of such $r$ if (i) the axis of symmetry $\left(u_{min}-\frac{\xi(\hat{s})}{u_{min}}\right)/2>0$ and (ii) the discriminant $\Delta = \left(u_{min}-\frac{\xi(\hat{s})}{u_{min}}\right)^2-4\xi(s-\hat{s}) > 0$. These two conditions are exactly \eqref{eqn:cond1} and \eqref{eqn:cond2}.

%By now, the satisfaction of \eqref{eqn:cond1} and \eqref{eqn:cond2} gives an interval of $r$. But we are interested in the smallest possible value of $r$, which is denoted by $$\rho = \frac{\left(u_{min}-\frac{\xi(\hat{s})}{u_{min}}\right)-\sqrt{\Delta}}2$$ since it provides a better locality for load flow solution.
%
%In summary, when the conditions \eqref{eqn:cond1}\eqref{eqn:cond2} are satisfied, $\tilde{G}$ is a self-mapping of $u$ on
%\[
%\tilde{\mathcal{D}} = \{u: \abs{u_i-\hat{u}_i}\leq \rho\}
%\]
By now, the satisfaction of \eqref{eqn:cond1} and \eqref{eqn:cond2} gives an interval of $r$. But we are interested in the smallest possible value of $r$, which is given by $\rho = \left(\left(u_{min}-\frac{\xi(\hat{s})}{u_{min}}\right)-\sqrt{\Delta}\right)/2$ since it provides a better description of locality for load-flow solution. This completes the proof of the Lemma.
\end{proof}
\begin{rem}
Equivalently, $G$ is a self-mapping of $v$ on $\mathcal{D}$.
\end{rem}

\subsection{Proof of contraction mapping}

\begin{lem}
Suppose that the pair $(\hat{v},\hat{s})$ and the complex power $s$ satisfy \eqref{eqn:cond1}
%$$\xi(\hat{s}) < u_{min}^2$$
and \eqref{eqn:cond2}.
%$$\Delta = \left(u_{min}-\frac{\xi(\hat{s})}{u_{min}}\right)^2-4\xi(s-\hat{s}) > 0.$$
Then $\tilde{G}$ is a contraction mapping of $u$ on the metric space $(\tilde{\mathcal{D}},\tilde{d})$, where $\tilde{\mathcal{D}}$ is given in \eqref{eqn:D_tilde} and
$\tilde{d}$ is defined by the $\ell_\infty$ norm.
\end{lem}

\begin{proof}
%From Lemma 1, we have $\tilde{G}(u)$ a self-mapping on closed set $\tilde{\mathcal{D}}$. 
As $\tilde{\mathcal{D}}$ is a convex set, there exists a straight path connecting any two points $u^1$ and $u^2$ in $\tilde{\mathcal{D}}$. Parameterize the path and denote it by $b$:  $b(t)=u^1+t(u^2-u^1)$ for $t \in [0,1]$. %Let $d\tilde{G}\left(b(t)\right)/dt$ denote the derivative of $\tilde{G}$, 
Then, we have the relation:
\begin{scriptsize}
\begin{align}
  \begin{split}
    \| \tilde{G}(u^2)-\tilde{G}(u^1) \|_\infty = & \| \tilde{G}(b(1))-\tilde{G}(b(0)) \|_\infty
    = \| \int_0^1 \frac{d\tilde{G}\left(b(t)\right)}{dt} dt \|_\infty. \nonumber
  \end{split}
\end{align}
\end{scriptsize}
By triangular inequality, it holds that
\begin{equation} \label{eqn:contrstart}
\| \tilde{G}(u^2)-\tilde{G}(u^1) \|_\infty \leq \int_0^1 \| \frac{d\tilde{G}\left(b(t)\right)}{dt} \|_\infty dt.
\end{equation}
We view $\mathbb{C}^N$ as an abstract vector space on $\mathbb{R}$ (i.e., of dimension $2N$), equipped with the norm $\|(z_1,...,z_N)\|_\infty\triangleq\max_{i=1}^N\abs{z_i}$. Note that this is a norm when we view $\mathbb{C}^N$ either as a $\mathbb{C}$-\emph{vector} space or an $\mathbb{R}$-\emph{vector} space. As shown in \cite{MatDiff},
\begin{equation}
\tilde{G}(b+h)=\tilde{G}(b)+\tilde{G}'(b)\cdot h+\|h\|_\infty\varepsilon(h) \quad \forall h\in \mathbb{C}^N, \nonumber
\end{equation}
where $\tilde{G}'(b): \mathbb{C}^N\rightarrow\mathbb{C}^N$, the differential operator of $\tilde{G}$ at $b$, is an $\mathbb{R}$-\emph{linear} operator, and ``$\cdot$'' denotes the action of this operator.
Then for the $\tilde{G}$ defined in \eqref{eqn:equiv_PFE_u}, we have
\begin{small}
\begin{equation}
\tilde{G}'(b)\cdot h=-W^{-1}Y_{LL}^{-1}\overline{W}^{-1}\diag\left(\frac{\overline{s}_1}{\overline{b}_1^2},...,\frac{\overline{s}_N}{\overline{b}_N^2}\right)\overline{h}. \nonumber
\end{equation}
\end{small}
So that, we continue the derivation in \eqref{eqn:contrstart} and obtain
% \int_0^1 \| \frac{d\tilde{G}\left(b(t)\right)}{dt} \|_\infty dt
%=
\begin{scriptsize}
\begin{align} \label{eqn:contrmid1}
  \begin{split}
    & \| \tilde{G}(u^2)-\tilde{G}(u^1) \|_\infty
    \leq \int_0^1 \| \tilde{G}'\left(b(t)\right)\cdot\frac{db(t)}{dt} \|_\infty dt \\
    = & {\int_0^1} \| W^{-1}Y_{LL}^{-1}\overline{W}^{-1} {\diag\left(\frac{\overline{s}_1}{\overline{b}_1^2(t)},...,\frac{\overline{s}_N}{\overline{b}_N^2(t)}\right)}\frac{d\overline{b}(t)}{dt} \|_\infty dt \\
    \leq & \int_0^1\| W^{-1}Y_{LL}^{-1}\overline{W}^{-1}\diag\left(\frac{\overline{s}_1}{\overline{b}_1^2(t)},...,\frac{\overline{s}_N}{\overline{b}_N^2(t)}\right) \|_\infty \| u^2-u^1 \|_\infty dt.
  \end{split}
\end{align}
\end{scriptsize}
Since $b(t)$ is always in $\tilde{\mathcal{D}}$, we have $\abs{b_i(t)}\geq u_{min}-\rho$. Then, by sub-multiplicativity of matrix norm, there is
\begin{scriptsize}
\begin{align} \label{eqn:contrmid2}
  \begin{split}
    & \| W^{-1}Y_{LL}^{-1}\overline{W}^{-1}\diag\left(\frac{\overline{s}_1}{\overline{b}_1^2(t)},...,\frac{\overline{s}_N}{\overline{b}_N^2(t)}\right) \|_\infty \\
    \leq & \| W^{-1}Y_{LL}^{-1}\overline{W}^{-1}\diag\left(\overline{s}_1,...,\overline{s}_N\right) \|_\infty
    \| \diag \left( \overline{b}_1^2(t),...,\overline{b}_N^2(t) \right)^{-1} \|_\infty \\
    \leq & \frac{\xi(s)}{(u_{min}-\rho)^2}.
  \end{split}
\end{align}
\end{scriptsize}
Further, observe that from \eqref{eqn:cond2},
\begin{small}
\begin{align}
  \begin{split}
    \Delta= & (u_{min}-\frac{\xi(\hat{s})}{u_{min}})^2-4\xi(s-\hat{s})\\
    = & (u_{min}+\frac{\xi(\hat{s})}{u_{min}})^2 - 4(\xi(\hat{s})+\xi(s-\hat{s}))
    > 0. \nonumber
  \end{split}
\end{align}
\end{small}
Hence, we have
\begin{small}
\begin{align} \label{eqn:contrfin}
  \begin{split}
    & \xi(s)= \| W^{-1}Y_{LL}^{-1}\overline{W}^{-1}\diag\left(\overline{s}\right) \|_\infty \\
    = & \| W^{-1}Y_{LL}^{-1}\overline{W}^{-1}\left(\diag\left(\overline{\hat{s}}\right)+\diag\left(\overline{s}-\overline{\hat{s}}\right)\right) \|_\infty \\
    \leq & \| W^{-1}Y_{LL}^{-1}\overline{W}^{-1}\diag\left(\overline{\hat{s}}\right) \|_\infty+ \\
    & \| W^{-1}Y_{LL}^{-1}\overline{W}^{-1}\diag\left(\overline{s}-\overline{\hat{s}}\right) \|_\infty \\
    = & \xi(\hat{s})+\xi(s-\hat{s})
    < \left(\frac{u_{min}+\frac{\xi(\hat{s})}{u_{min}}}{2}\right)^2 \\
    < & \left(\frac{u_{min}+\frac{\xi(\hat{s})}{u_{min}}+\sqrt{\Delta}}{2}\right)^2
    = (u_{min}-\rho)^2.
  \end{split}
\end{align}
\end{small}
Thus, by combining \eqref{eqn:contrmid1}, \eqref{eqn:contrmid2} and \eqref{eqn:contrfin}, we obtain
\begin{scriptsize}
\begin{align}
  \begin{split}
    & \| \tilde{G}(u^2)-\tilde{G}(u^1) \|_\infty \\
    \leq & \int_0^1\| W^{-1}Y_{LL}^{-1}\overline{W}^{-1}\diag\left(\frac{\overline{s}_1}{\overline{b}_1^2(t)},...,\frac{\overline{s}_N}{\overline{b}_N^2(t)}\right) \|_\infty \| u^2-u^1 \|_\infty dt\\
    \leq & \frac{\xi(s)}{(u_{min}-\rho)^2}\| u^2-u^1 \|_\infty
    < \| u^2-u^1 \|_\infty \nonumber
  \end{split}
\end{align}
\end{scriptsize}
which completes the proof of the Lemma.
\end{proof}
\begin{rem}
Equivalently, $G$ is a contraction mapping of $v$ on metric space $(\mathcal{D},d)$ where $d$ is defined by weighted vector norm $\ell_{W,\infty}$ such that $\| v \|_{W,\infty}\triangleq \|W^{-1}v \|_\infty$.
\end{rem}

\section{Numerical Illustration} \label{sec:num}
The proposed conditions have been tested through a large number of experiments on the basis of IEEE models \cite{NumExp}. Due to space limitations, we show the numerical result of one experiment on an IEEE 13-feeder model whose structure is illustrated as following in Fig.\ref{fig:IEEEModel}. We adjust it by assuming all power lines are of same type but different length. The model parameters are taken as typical values for medium-voltage cables as in \cite{NumExp2}.

\begin{figure}[h!]
\begin{center}
\includegraphics[scale=0.6]{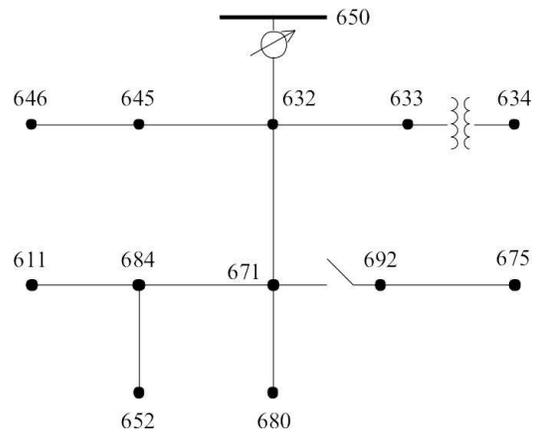}
\caption{IEEE 13-feeder grid.}
\label{fig:IEEEModel}
\end{center}
\end{figure}
The power components of the known solution $(\hat{s}, \hat{v})$ are given in Table \ref{table:Tab2}; voltage magnitudes are shown on Fig.\ref{fig:CompNR}. 
%The specific power injections are shown in Table \ref{table:Tab2}. 
For better expression, first re-number all the nodes. Then take the power injection $\hat{s}=\hat{p}+j\hat{q}$ with normalization base $5$MVA for Power and $4.16/\sqrt{3}=2.4$kV for Voltage (which is also the voltage of the slack bus). 
%Voltages are not shown in detail for lack of space, the voltage magnitudes are in the range $[VV, VV]$ p.u. and the angles are in the range $[RR, RR]$ radians.
%In the first subsection, we illustrate the proposed method on one example;
%, where $s_{j}=\hat{s}_{j}e_{j}$. 
%in the second subsection, we do a continuation power flow analysis \cite{CPFLOW}.
\begin{table}[h!]
\caption{Key Parameters}
\label{table:Tab2}
\begin{small}
\begin{center}
\begin{tabular}{ccccc}
\noalign{\global\arrayrulewidth0.05cm}
\hline
\noalign{\global\arrayrulewidth0.4pt}
Index & $\hat{p}$(MW) & $\hat{q}$(Mvar) & $|\hat{s}|$(MVA) & e\\
%\noalign{\global\arrayrulewidth0.05cm}
\hline
%\noalign{\global\arrayrulewidth0.4pt}
$632\rightarrow1$ & -0.48 & -0.32 & 0.58 & 1 \\
$633\rightarrow2$ & 1.28 & 0.96 & 1.60 & 1.05 \\
$634\rightarrow3$ & -0.72 & -0.48 & 0.87 & 0.95 \\
$645\rightarrow4$ & 0.96 & 0.8 & 1.25 & 1.03 \\
$646\rightarrow5$ & -0.96 & -0.8 & 1.25 & 1.01 \\
$671\rightarrow6$ & 0.64 & 0.48 & 0.80 & 1.05 \\
$692\rightarrow7$ & -0.8 & -0.48 & 0.93 & 0.97 \\
$675\rightarrow8$ & 0.64 & 0.48 & 0.80 & 1.04 \\
$684\rightarrow9$ & -0.64 & -0.48 & 0.80 & 0.99 \\
$611\rightarrow10$ & 0.32 & 0.24 & 0.4 & 1 \\
$680\rightarrow11$ & -0.48 & -0.32 & 0.58 & 1 \\
$652\rightarrow12$ & 0.32 & 0.24 & 0.4 & 1.05 \\
\noalign{\global\arrayrulewidth0.05cm}
\hline
\noalign{\global\arrayrulewidth0.4pt}
\end{tabular}
\end{center}
\end{small}
\end{table}

%\vspace*{-0.8cm}
\subsection{Illustration of Main Theorem}

Here, for illustration purpose, we apply Theorem~\ref{thm:main} to test the candidate power injection $s$, where $s_{j}=\hat{s}_{j}e_{j}$ with $\hat{s}$ and $e$ as in Table \ref{table:Tab2}. The computed results are shown in Table \ref{table:Tab3}. It is easy to check that the conditions in Theorem 1 are satisfied. In contrast, note that $\xi(s)=0.5770>0.25$, i.e. the method and conditions given in \cite{Bolo} and \cite{ImproveBolo} do not work in this case.

\begin{table}[h!]
\caption{Computed Results}
\label{table:Tab3}
\begin{small}
\begin{center}
\begin{tabular}{ccccc}
\noalign{\global\arrayrulewidth0.05cm}
\hline
\noalign{\global\arrayrulewidth0.4pt}
$\xi(\hat{s})$ & $\xi(s-\hat{s})$ & $\xi(s)$ & $u_{min}$ & $\rho$ \\
%\noalign{\global\arrayrulewidth0.05cm}
\hline
%\noalign{\global\arrayrulewidth0.4pt}
0.5692 & 0.0164 & 0.5770 & 1.0050 & 0.0412 \\
\noalign{\global\arrayrulewidth0.05cm}
\hline
\noalign{\global\arrayrulewidth0.4pt}
\end{tabular}
\end{center}
\end{small}
\end{table}

In Fig.\ref{fig:Ddomain}, the red circle is of radius $\rho=0.0412$ and represents $\mathcal{D}$ for one coordinate (here for instance, select Node 8).

\begin{figure}[h!]
\begin{center}
\includegraphics[scale=0.28]{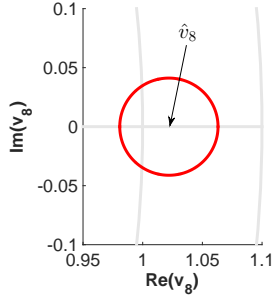}
\caption{The domain $\mathcal{D}$ for one coordinate (Node 8).}
\label{fig:Ddomain}
\end{center}
\end{figure}

In Fig.\ref{fig:CompNR}, the solved voltage magnitudes are shown. In the same figure, the Newton-Raphson method is used for checking the result. It is well-observed that the method gives out the same solution as Newton-Raphson method. Actually, all the solution coordinates lie in the domain given by our theorem.

\begin{figure}[h!]
\begin{center}
\includegraphics[height=5cm,width=9cm]{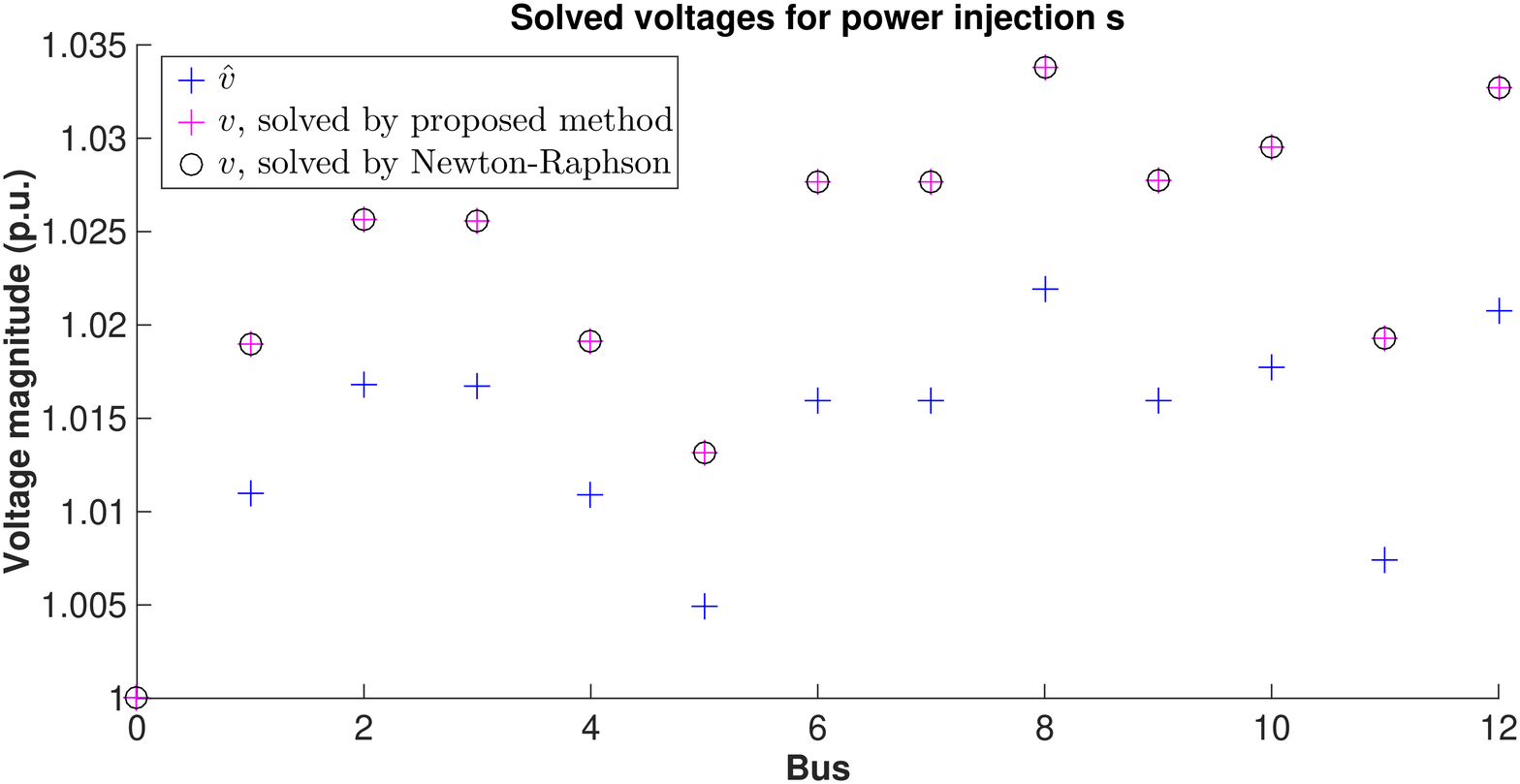}
\caption{The voltages of power injection $\hat{s}$ and the computed voltages of power injection $s$.}
\label{fig:CompNR}
\end{center}
\end{figure}
\begin{figure}[h!]
\begin{center}
\includegraphics[height=3.6cm,width=8.6cm]{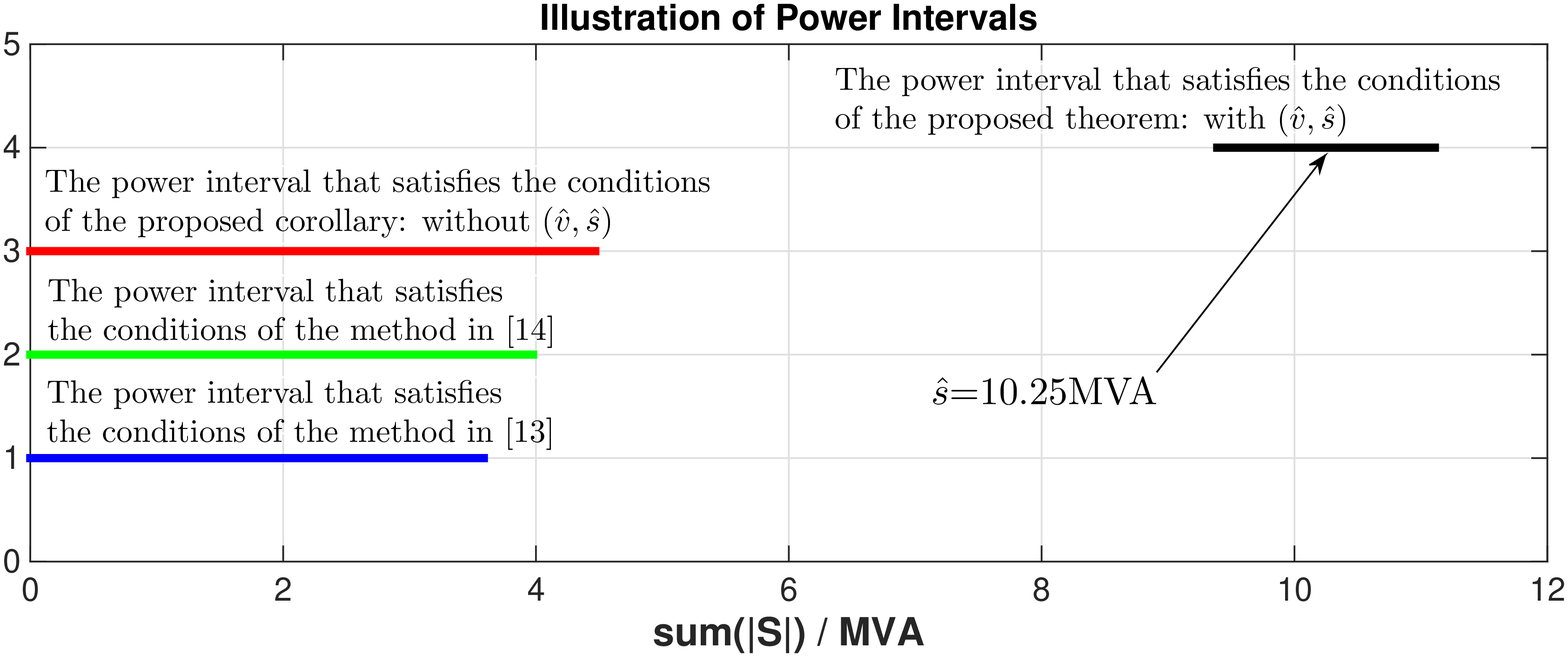}
\caption{Intervals of power injection that satisfy the conditions of the proposed theorem, the proposed corollary, the method in \cite{Bolo} and the method in \cite{ImproveBolo}.}
\label{fig:CPFLOW}
\end{center}
\end{figure}

\subsection{Continuation Power Flow analysis}
In this subsection, we illustrate the range of power injections that are allowed and provided by our theorems, using ``continuation power flow analysis'' \cite{CPFLOW}. To this end, we do not take the candidate power injections $s$ from Table \ref{table:Tab2} but instead we scale them from $\hat{s}$. Specifically, let $s=\kappa \frac{\hat{s}}{{\parallel \hat{s} \parallel}_1}$ with $\kappa \in [0,\infty)$ MVA. In other words, the scaling factor $\kappa=\sum_{i=1}^{N}|s_{i}|$ is the sum of all apparent power injections. Then,
%\subsubsection{With ($\hat{v}$,$\hat{s}$)}

\textit{1) With ($\hat{v}$,$\hat{s}$)}: By applying the conditions of the proposed main theorem, the black interval is obtained in Fig.\ref{fig:CPFLOW}. For all the summed power $\kappa$ in this interval, our conditions \eqref{eqn:cond1} and \eqref{eqn:cond2} are satisfied.
%\subsubsection{Without ($\hat{v}$,$\hat{s}$)}

\textit{2) Without ($\hat{v}$,$\hat{s}$)}: Similarly, we can obtain the red interval by applying the conditions of the proposed corollary, the green interval by conditions of the method in \cite{ImproveBolo}, and the blue interval by conditions of the method in \cite{Bolo}. In this example, it is clear that the power interval provided by the proposed method (i.e., red interval) covers the power intervals provided by methods in \cite{Bolo} and \cite{ImproveBolo} (i.e., green and blue intervals). In other words, the proposed method is (strictly) stronger than the methods in \cite{Bolo} and \cite{ImproveBolo}.
\begin{rem}
Here, the $\Lambda$ for the method in \cite{ImproveBolo} is chosen as suggested in \cite{ImproveBolo} $\Lambda_k=1/\max_h|(W^{-1}Y_{LL}^{-1}\overline{W}^{-1})_{hk}|$.
\end{rem}

\label{sec-illu}
\section{Conclusion} \label{sec:conc}
We have provided explicit sufficient conditions that guarantee the existence and uniqueness of the feasible load-flow solution for distribution networks with generic topology modeled using their positive sequence equivalents. Our findings improve on all previously known results. The whole theory has been verified in IEEE benchmark grids.

 %In addition, through the comparison with the other existent results from both theoretical and numerical perspectives, the stronger performance of proposed theory is solidly illustrated.
%%
The proposed method is of practical use, as it can easily be deployed in applications for microgrids and distribution networks that require solving load-flows in real time.

%As another important aspect, the proposed method is of practical usage since the complexities for checking conditions are considerably low. Specifically, it is linear in all radial networks where the shunt elements are negligible. This property enables system devices (e.g., agents) to check the conditions in advance and determine whether or not to solve the load-flow problem, which is crucial for real-time operation.

%In summary, the proposed method opens a way of performance-guaranteed real-time load-flow computation. 
We plan to extend the results to more general three-phase networks in a subsequent paper. 
%\vspace{-0.5cm}
\appendices
\section{Invertibility of $Y_{LL}$} \label{sec:app}
In circuit theory \cite{Kuh}, there are already results on the invertibility of a full admittance matrix which includes the ground as one node. However, these results do not directly apply to $Y_{LL}$, which is only a sub-matrix of the nodal admittance matrix $Y$ that does not contain ground node. Having considered this fact, we give the proof of the invertibility of $Y_{LL}$ in this appendix.
%But for the nodal admittance matrix in power system, the ground node is not included. Moreover, $Y_{LL}$ is only a sub-matrix of the nodal admittance matrix. Thus the results in references like \cite{Kuh} cannot be applied to show the invertibility of $Y_{LL}$. Having considered this fact, we give the proof of the invertibility of $Y_{LL}$ in this appendix. And as a byproduct, we also discuss about the invertibility of nodal admittance matrix $Y$.
It is worth noticing that the proof does not require the network to be radial.
\subsection{Modeling and the Admittance Matrix}
For the non-transformer connection (e.g., transmission lines) between node $i$ and $j$, the $2\times2$ longitudinal admittance matrix is
\[
\left[
\begin{array}{cc}
y_{ij} & -y_{ij} \\
-y_{ij} & y_{ij}
\end{array}
\right]
\]
where $y_{ij}$ (equal to $y_{ji}$) is the summed admittance of all power lines going directly from node $i$ to node $j$.

For the transformer connection between node $i$ and $j$, without loss of generality, let node $i$ be connected to the primary side of this transformer and node $j$ be at the secondary side, the  $2\times2$ admittance matrix is given as
\[
\left[
\begin{array}{cc}
y_{ij}^t & -y_{ij}^tK_{ij}^{-1} \\
-y_{ij}^t\overline{K_{ij}}^{-1} & y_{ij}^t|K_{ij}|^{-2}
\end{array}
\right]
\]
where $y_{ij}^t$ is the equivalent aggregated admittance on the primary side, complex number $K_{ij}$ is the ratio. Reciprocally, we can denote $y_{ij}^t|K_{ij}|^{-2}$ by $y_{ji}^t$ which is the equivalent aggregated admittance on the secondary side, and $K_{ij}^{-1}$ by $K_{ji}$ which is the inverse ratio.
%\begin{note}
%For the transformer with core loss, we can apply $Y-\Delta$ transformation to obtain the equivalent aggregated longitudinal admittance $y_{ij}^t$. As by-product of the transformation, there are two shunt-like elements assigned to node $i$ and $j$ respectively. We should notice that all the three resulting elements have positive real part, which will be useful in the following proof.
%\end{note}
Now, the terms in a general admittance matrix $Y$ including shunt elements can be explicitly written as
\[
Y_{ij}=
\left\{
\begin{array}{ll}
-y_{ij} & j\in\mathcal{N}(i) \\
-y_{ij}^tK_{ij}^{-1} & j\in\mathcal{N}^t(i) \\
0 & \textrm{otherwise}
\end{array}
\right.
\]
and
$$Y_{ii}=y_{ii}^{shunt}+\sum_{j\in\mathcal{N}(i)}y_{ij}+\sum_{j\in\mathcal{N}^t(i)}y_{ij}^t,$$
where $\mathcal{N}(i)$ is the set of nodes that have direct non-transformer connections with node $i$, and $\mathcal{N}^t(i)$ is the set of nodes that have direct transformer connections with node $i$. Here, $y_{ii}^{shunt}$ is the sum of shunt elements around node $i$.% resulted by transmission lines and shunt-like elements resulted by core-loss transformers surrounding node $i$.
\subsection{The Invertibility}
If the grid is viewed as a graph where buses are vertices and power lines are edges, then a new graph can be generated by eliminating node 0. Suppose that the new graph has $c$ connected components, then by carefully re-numbering each node, $Y_{LL}$ can be written as a $c$-block diagonal matrix. In this way, $Y_{LL}$ is invertible iff all blocks are invertible. Thus, if we can show an arbitrary one of these components invertible, then the invertibility of $Y_{LL}$ is proved. Thus, without loss of generality, assume that the new graph itself be one connected component.

First, denote this undirected graph as $\mathcal{G}=(\mathcal{V},\mathcal{E})$. In addition, let $\mathcal{V}^{slack}\subseteq \mathcal{V}$ be the set of nodes that are originally connected to the slack bus; $\mathcal{G}^{t}=(\mathcal{V}^t,\mathcal{E}^t)$ be the subgraph that contains all the transformer edges and corresponding endpoints; $\mathcal{G}_m=(\mathcal{V}_m,\mathcal{E}_m)$, $m\in\{1,...,M\}$ be all the $M$ connected components in $(\mathcal{V},\mathcal{E}\setminus\mathcal{E}^t)$.

Let $x$ be an N-by-1 vector such that $Y_{LL}x=0$, and for all $i\in\mathcal{V}^{slack}$ define
\[
\tilde{y}_{i0}=
\left\{
\begin{array}{ll}
y_{i0} & \textrm{non-transformer connection} \\
y_{i0}^t & \textrm{transformer connection}
\end{array}
\right.
\]
Then, we have
\begin{footnotesize}
\begin{align}
  \begin{split}
  	& x^{H}Y_{LL}x=\sum_{i,j\in\mathcal{V}}\overline{x}_i(Y_{LL})_{ij}x_j\\
  	= & \sum_{i=1}^N\sum_{j:(i,j)\in\mathcal{E}\setminus\mathcal{E}^t}y_{ij}\overline{x}_i(x_i-x_j)+\sum_{i=1}^N\sum_{j:(i,j)\in\mathcal{E}^t}y_{ij}^t\overline{x}_i(x_i-K_{ij}^{-1}x_j) \\
  	& +\sum_{i\in\mathcal{V}^{slack}}\tilde{y}_{i0}|x_i|^2+\sum_{i\in\mathcal{V}}y_{ii}^{shunt}|x_i|^2 \nonumber
  \end{split}
\end{align}
\end{footnotesize}
For the first term, we have
\begin{footnotesize}
\begin{align}
  \begin{split}
  	& \sum_{i=1}^N\sum_{j:(i,j)\in\mathcal{E}\setminus\mathcal{E}^t}y_{ij}\overline{x}_i(x_i-x_j)\\
  	= & \sum_{i=1}^N\sum_{j>i:(i,j)\in\mathcal{E}\setminus\mathcal{E}^t}y_{ij}\overline{x}_i(x_i-x_j)+\sum_{i=1}^N\sum_{j<i:(i,j)\in\mathcal{E}\setminus\mathcal{E}^t}y_{ij}\overline{x}_i(x_i-x_j) \\
  	= & \sum_{i=1}^N\sum_{j>i:(i,j)\in\mathcal{E}\setminus\mathcal{E}^t}y_{ij}\overline{x}_i(x_i-x_j)+\sum_{i=1}^N\sum_{j>i:(i,j)\in\mathcal{E}\setminus\mathcal{E}^t}y_{ji}\overline{x}_j(x_j-x_i)\\
  	= & \sum_{i=1}^N\sum_{j>i:(i,j)\in\mathcal{E}\setminus\mathcal{E}^t}\left(y_{ij}\overline{x}_i(x_i-x_j)+y_{ij}\overline{x}_j(x_j-x_i)\right)\\
  	= & \sum_{i<j:(i,j)\in\mathcal{E}\setminus\mathcal{E}^t}y_{ij}|x_i-x_j|^2 \nonumber
  \end{split}
\end{align}
\end{footnotesize}
Similarly, for the second term, we have
\begin{scriptsize}
\begin{align}
  \begin{split}
  	& \sum_{i=1}^N\sum_{j:(i,j)\in\mathcal{E}^t}y_{ij}^t\overline{x}_i(x_i-K_{ij}^{-1}x_j)\\
  	= & \sum_{i=1}^N\sum_{j>i:(i,j)\in\mathcal{E}^t}y_{ij}^t\overline{x}_i(x_i-K_{ij}^{-1}x_j)+\sum_{i=1}^N\sum_{j<i:(i,j)\in\mathcal{E}^t}y_{ij}^t\overline{x}_i(x_i-K_{ij}^{-1}x_j) \\
  	= & \sum_{i=1}^N\sum_{j>i:(i,j)\in\mathcal{E}^t}y_{ij}^t\overline{x}_i(x_i-K_{ij}^{-1}x_j)+\sum_{i=1}^N\sum_{j>i:(i,j)\in\mathcal{E}^t}y_{ji}^t\overline{x}_j(x_j-K_{ji}^{-1}x_i)\\
  	= & \sum_{i=1}^N\sum_{j>i:(i,j)\in\mathcal{E}^t}(y_{ij}^t\overline{x}_i(x_i-K_{ij}^{-1}x_j)+y_{ij}^t\overline{K}_{ij}^{-1}\overline{x}_j(K_{ij}^{-1}x_j-x_i))\\
  	= & \sum_{i<j:(i,j)\in\mathcal{E}^t}y_{ij}^{t}|x_i-K_{ij}^{-1}x_j|^2 \nonumber
  \end{split}
\end{align}
\end{scriptsize}
So that,
\begin{footnotesize}
\begin{align}
  \begin{split}
  	& x^{H}Y_{LL}x\\
  	= & \sum_{i<j:(i,j)\in \mathcal{E}\setminus \mathcal{E}^t}y_{ij}|x_i-x_j|^2+\sum_{i<j:(i,j)\in\mathcal{E}^t}y_{ij}^{t}|x_i-K_{ij}^{-1}x_j|^2 \\
  	& +\sum_{i\in\mathcal{V}^{slack}}\tilde{y}_{i0}|x_i|^2+\sum_{i\in\mathcal{V}}y_{ii}^{shunt}|x_i|^2=0 \nonumber
  \end{split}
\end{align}
\end{footnotesize}
Since $\Re{\tilde{y}_{i0}}>0$ for all $i\in\mathcal{V}^{slack}$, 
$\Re{y_{ii}^{shunt}}$ non-negative for all $i\in\mathcal{V}$, 
and $\Re{y_{ij}},\Re{y_{ij}^t}>0$ for all $i,j$ s.t. $(i,j)\in\mathcal{E}$, we have
\begin{enumerate}
\item $x_i=0$ for all $i\in\mathcal{V}^{slack}$;
\item $x_i=x_j$ for all $i,j\in\mathcal{V}_m$ given any $m\in\{1,...,M\}$;
\item $x_i=K_{ij}^{-1}x_j$ for all $i,j$ s.t. $(i,j)\in\mathcal{E}^t.$
\end{enumerate}
Because $\mathcal{G}$ is connected, it can be obtained that
\begin{itemize}
\item By above 1 and 2, there exists at least one $m$ s.t. $x_i=0$ for all $i\in\mathcal{V}_m$.
\item By 2 and 3, the zero value will propagate throughout $\mathcal{G}$.
\end{itemize}
Thus, the vector $x$ must be a zero vector, which implies $Y_{LL}$ has a trivial null space and hence is invertible.

%\begin{rem}
%The same technique can be applied to show a general nodal admittance matrix $Y$ is invertible whenever there is shunt element in at least one line. This is because $\mathbbm{1}$ is the only basis of null space for the real part of nodal admittance matrix, but it is not the basis for the null space of the imaginary part of nodal admittance matrix when shunt element exists. Therefore, the null space of the nodal admittance matrix is trivial in such cases.
%\end{rem}

%\section{References}
\bibliographystyle{IEEEtran}
\bibliography{congrefs}
%\appendices
%\input{./sections/appendix}
\end{document}